\documentclass[12pt]{article}

\usepackage{amssymb}
\usepackage{amsmath,amsthm,amsfonts}
\usepackage[latin1]{inputenc}
\usepackage{graphicx}
\usepackage{hyperref}
\usepackage{enumerate}
\usepackage{tikz}
\usepackage{color,url}
\usepackage{cite}
\usepackage{mathtools}

\newtheorem{theorem}{Theorem}[section]

\newtheorem{corollary}[theorem]{Corollary}

\newtheorem{lemma}[theorem]{Lemma}
\newtheorem{problem}{Problem}
\newtheorem{proposition}[theorem]{Proposition}

\newcommand{\cp}{\,\square\,}

\newcommand{\cF}{\mathcal{F}}
\DeclareMathOperator{\sab}{\mathrm{sab}}
\DeclareMathOperator{\height}{\mathrm{h}}

\begin{document}

\title{Maker-Breaker Sabotage Game}

\author{
Marko Jakovac 
$^{a,b,}$\thanks{\texttt{marko.jakovac@um.si}}
\and Sandi Klav\v zar $^{a,b,c,}$\thanks{\texttt{sandi.klavzar@fmf.uni-lj.si}}
\and Mirjana Mikala\v{c}ki $^{d,}$\thanks{\texttt{mima@dmi.uns.ac.rs}}
\and Andrej Taranenko $^{a,b,}$\thanks{\texttt{andrej.taranenko@um.si}}
}
\maketitle

\begin{center}
$^a$ Faculty of Natural Sciences and Mathematics, University of Maribor, Slovenia \\
\medskip
$^b$ Institute of Mathematics, Physics and Mechanics, Ljubljana, Slovenia\\
\medskip
$^c$ Faculty of Mathematics and Physics, University of Ljubljana, Slovenia\\
\medskip
$^d$ Department of Mathematics and Informatics, Faculty of Sciences, University of Novi Sad, Serbia\\
	\medskip
\end{center}

\begin{abstract}
The Maker-Breaker sabotage game is played on a graph $G$ by Runner and Blocker. They play in turns, Runner first moves along a not yet traversed edge from her current position, afterwards Blocker removes one edge. The goal of Runner is to visit as many vertices of $G$ as possible, Blocker's goal is opposite. Assuming that both players use optimal strategies, the number of vertices visited by Runner determines an invariant called the sabotage number ${\rm sab}(G)$ of $G$. 
A formula for the sabotage number of an arbitrary tree is proved which can be evaluated in polynomial time. For a unicyclic graph $G$ it is proved that $\sab(G)\in \{\sab^-(G), \sab^-(G)+1\}$, where $\sab^-(G)$ is the lower sabotage number of $G$. The sabotage number of a bridgeless subcubic graph is sharply bounded from the above by the maximum girth. The sabotage number is also bounded for complete bipartite graphs and generalized Sierpi\'nski graphs, and determined exactly in some special cases. 
\end{abstract}

\noindent
{\bf Keywords:} sabotage game; Maker-Breaker positional game; tree; subcubic graph; complete bipartite graph; Sierpi\'nski graph   \\

\noindent
{\bf AMS Subj.\ Class.\ (2020)}: 05C57

\section{Introduction}


Two decades ago, van Benthem introduced sabotage games~\cite{bentham-2005}, which involve two players named Runner and Blocker. The aim of Runner is to travel between two given vertices in a network, while the goal of Blocker is to prevent Runner from achieving her goal. The tool available to Blocker for this purpose is to remove (or destroy) edges of the network. For more information on these games, see~\cite{bentham-2026, kvasov-2016, mierzewski-2026, raju-2026} and the references therein. It should be noted that the corresponding sabotage modal logic was also introduced in the seminal  article~\cite{bentham-2005}, which has led to numerous studies, see, for example,~\cite{abou-2024, aucher-2018}.

Maker-Breaker games are the most studied representatives of positional games---combinatorial games of complete information played by two players that cover a wide range of games, ranging from popular games like Tic-Tac-Toe or Hex, to purely abstract games played on various structures. In Maker-Breaker games, given a finite board $X$ and a family of winning sets $\cF \subseteq 2^X$, the players, called \emph{Maker} and \emph{Breaker}, take turns claiming the elements of $X$, one element per turn, until all of them are claimed. Maker wins if, by the end of the game, she has claimed all the elements of some $f\in \cF$, and Breaker wins otherwise. No draw is possible. The study of these games started with the seminal papers~\cite{chvatal-1978,erdos-1973} and has since attracted great attention, motivating the study of Positional games as a whole with various modifications of the rules. For a comprehensive overview, we refer the reader to the books~\cite{beck-2008,hefetz-2014}. Over time, many restrictions have been applied to how Maker can play (see, e.g.,~\cite{chvatal-1978,clemens-2021,espig-2015,london-2018,seress-1999}) in order to give Breaker more power, since in many standard graph games it is very easy for Maker to win the game. Also, the quantitative version of Maker-Breaker games was introduced in~\cite{dowden-2019}, and further studied in~\cite{boriboon-2022,raty-2020} where Maker, called \textit{Toucher}, aims at touching as many vertices as possible by claiming edges, while Breaker's goal was opposite, i.e.\ to isolate as many vertices possible. 

In this paper, we introduce and study the following Maker-Breaker sabotage game. We are given a connected graph $G = (V(G), E(G))$ with at least one edge, whose edge set serves as the board of the game, and two players, Runner and Blocker, with Runner taking the role of Maker and Blocker taking the role of Breaker. Runner starts the game by moving along a selected edge. After that, the two players play in turns. When it is Blocker's turn, he removes (deletes) one edge of the graph. When it is Runner's turn, she moves along an available edge incident to her current vertex that she has not traversed before. The goal of Runner is to visit as many vertices of $G$ as possible, while the goal of Blocker is opposite: to minimize the number of vertices Runner visits. If both players play optimally, the number of vertices visited by Runner is a graph invariant. We call it the \emph{sabotage number} of $G$ and denote it by $\sab(G)$. As the Maker-Breaker sabotage game is the unique game studied here, for this paper we simplify the terminology and briefly call it the {\em sabotage game}. 

In the above-defined sabotage game, the goal of Runner is effectively to create a trail (a walk in which no edge is repeated) that is as long as possible. Positional games focusing on the step-by-step formation of trails have been studied previously under different rule sets. For instance, our game shares underlying concepts with Erd\H{o}s's Eulerian trail game~\cite{seress-1999}. Similarly, feedback games have recently garnered attention, having been investigated on Eulerian graphs~\cite{matsumoto-2023} and on 3-chromatic Eulerian triangulations of surfaces~\cite{higashi-2024}. While these games also involve navigating a graph to create a walk where no edge is repeated, the unique mechanism of the sabotage game---where Blocker actively deletes edges from the board rather than moving a shared token---introduces a fundamentally different dynamic.

We proceed as follows. In the next section, we list some definitions, determine the sabotage number for a few simple graph classes, characterize the triangle-free graphs with the sabotage number three, and introduce the lower sabotage number $\sab^-$. In Section~\ref{sec:trees}, we give a formula for the sabotage number of an arbitrary tree, which can be evaluated in polynomial time. We also prove that if $G$ is a unicyclic graph, then $\sab(G)\in \{\sab^-(G), \sab^-(G)+1\}$. In Section~\ref{sec:subcubic}, we bound from above the sabotage number of a bridgeless subcubic graph by its maximum edge girth of and demonstrate the sharpness of the bound. In Section~\ref{sec:additional}, we bound from above $\sab(K_{n,m})$, determine the sabotage number of classical Sierpi\'nski graphs, and bound it for generalized Sierpi\'nski graphs.  

\section{Preliminaries}
\label{sec:preliminaries}

In this section, we first provide a few additional definitions that we need. After that, we determine the sabotage number for some basic classes of graphs, characterize triangle-free graphs with the sabotage number equal to $3$, and introduce the lower sabotage number. For any graph theoretic notions not explicitly defined here, see \cite{west-2001}.

If $T$ is a tree and $u \in V(T)$, then by $T_u$ we denote the tree $T$ rooted at $u$. The height of a rooted tree $T_u$, denoted $\height(T_u)$, is the length of the longest path from the root $u$ to a leaf of $T_u$. A binary tree is a rooted tree in which each vertex has at most two children and each child of a vertex is designated as its left or right child. A complete binary tree of height $h$ is a binary tree in which all leaves are at distance $h$ from the root and every non-leaf vertex has exactly two children.

We will adopt the following convention. An edge $\{u,v\}$ will be shortly written as $uv$. Moreover, when such an edge represents an edge traversed by Runner in that direction, the first vertex in $uv$ is the first vertex traversed by Runner. 

For a statement $S$, we will use the Iverson bracket $[S]$, which has a value of 1 if $S$ is true and 0 otherwise. For a positive integer $k$, we use the notations $[k] = \{1,\dots, k\}$ and $[k]_0 = \{0,1,\dots,k-1\}$.  

We next determine the sabotage number of the following standard graph classes. 

\begin{proposition}\label{prp:path-cycle-complete}
If $G$ is a path, a cycle, or a complete graph, then 
\[
\sab(G) = \begin{dcases}
            2; & G \text{ is a path on } n\geq 2 \text{ vertices,} \\
            2; & G \text{ is a cycle on } n\geq 3 \text{ vertices,} \\
            n-1; & G \text{ is a complete graph on } n\geq 2 \text{ vertices.} \\
          \end{dcases}
\]
\end{proposition}

\begin{proof}
    The proofs for cycles and paths are trivial. For complete graphs, the cases $n=2$ and $n=3$ correspond to $P_2$ and $C_3$, respectively, giving $\sab(K_2) = 2$ and $\sab(K_3) = 2$. Assume $n \geq 4$. 

    We first prove the upper bound, $\sab(K_n) \leq n-1$. Blocker's strategy is to choose a specific vertex $u$ distinct from the two vertices of Runner's initial edge. In every subsequent turn, whenever Runner moves to a vertex $w$, Blocker deletes the edge $wu$, if it exists, or any other edge otherwise. Consequently, Runner can never visit $u$.

    To prove the lower bound, $\sab(K_n) \geq n-1$, we provide a strategy for Runner: at each step, if there is an available edge to an unvisited vertex, Runner traverses it; otherwise, she traverses any available edge to a previously visited vertex (such a move is called a \emph{backtracking} move).

    Suppose towards a contradiction that Runner follows this strategy and has no available edges to move from a vertex $v$ after visiting $k < n-1$ vertices (she is trapped at $v$). Let $U$ be the set of unvisited vertices ($|U|=n-k \geq 2$). 

    Let us examine the very first time in the game Runner was forced to make a backtracking move. Suppose this occurs at a vertex $x$ after she had visited $j$ distinct vertices ($j \le k \le n-2$). So far, Runner has traversed $j-1$ edges (forming a path) and Blocker has deleted $j-1$ edges. To force Runner to backtrack at $x$, Blocker must have deleted all $n-j$ edges connecting $x$ to the $n-j$ currently unvisited vertices. From now on, every time Runner makes a backtracking move to a previously visited vertex $y$ and is forced to backtrack again (or gets permanently trapped), Blocker must have deleted all available edges from $y$ to the set of unvisited vertices $U$. Because $|U| \ge 2$, it costs Blocker at least $2$ deletions strictly on edges incident to Runner's current vertex to prevent her from visiting an unvisited vertex. However, each backtracking move Runner makes only grants Blocker $1$ turn. Because Blocker started with $j-1$ deletions at Runner's first backtrack, he cannot have made enough deletions on other visited vertices she can backtrack to, to prevent Runner from visiting an unvisited vertex. Hence, Runner will eventually find an edge incident to an unvisited vertex to traverse, a contradiction that she is trapped at $k \le n-2$. Thus, $\sab(K_n) = n-1$.
\end{proof}

\begin{lemma}\label{lem:subgraph}
If $H$ is a connected subgraph of $G$, then $\sab(H)\le \sab(G)$.
\end{lemma}

\begin{proof}
Let Runner use an optimal strategy on $H$. We show that the same strategy also guarantees that Runner visits at least $\sab(H)$ vertices when the game is played on $G$.

Runner starts in $G$ exactly as she would start in $H$, and as long as the game involves only edges of $H$, she makes the same moves as prescribed by her optimal strategy on $H$. Assume now that at some stage of the game Blocker removes an edge from $E(G)\setminus E(H)$. Such a move has no effect on the availability of edges of $H$. From the point of view of Runner's strategy on $H$, this is even better than an arbitrary move of Blocker inside $H$, since no edge of $H$ has been removed. Thus, Runner pretends Blocker removed an arbitrary edge of $H$ and can still make the move prescribed by her strategy on $H$.

Consequently, regardless of how Blocker plays in $G$, Runner can simulate her optimal strategy from $H$ inside the subgraph $H$. Therefore, she visits at least $\sab(H)$ vertices of $H$, and hence at least $\sab(H)$ vertices of $G$. It follows that $\sab(G)\ge \sab(H)$ as claimed.    
\end{proof}

Note that the only connected graphs $G$ with $\sab(G) = 2$ are paths with at least two vertices and cycles. Indeed, if $G$ contains a vertex $w$ of degree at least three, then traversing an edge from a neighbor of $w$ to $w$ as the start of the game guarantees Runner to visit at least three vertices.  Graphs $G$ with $\sab(G) = 3$ can be characterized among triangle-free graphs as follows. 

\begin{proposition}
\label{prp:delta=3-triangle-free}
Let $G$ be a connected, triangle-free graph of order at least four. Then $\sab(G) = 3$ if and only if $\Delta(G) \ge 3$ and $G$ contains no subgraph isomorphic to one of the graphs from Fig.~\ref{fig:forbidden}.  
\end{proposition}

\begin{figure}[ht!]
    \centering
\begin{tikzpicture}[
    node style/.style={circle, draw, minimum size=8pt, inner sep=0pt, thick},
    edge style/.style={thick}, 
    scale=0.9
]

    \begin{scope}[xshift=0cm]
        \node[node style] (G1_1) at (0, 0) {};
        \node[node style] (G1_2) at (0, -1.5) {};
        \node[node style] (G1_3) at (-1, -3) {};
        \node[node style] (G1_4) at (1, -3) {};
        \node[node style] (G1_5) at (-1.5, -4.5) {};
        \node[node style] (G1_6) at (-0.5, -4.5) {};
        \node[node style] (G1_7) at (0.5, -4.5) {};
        \node[node style] (G1_8) at (1.5, -4.5) {};

        \draw[edge style] (G1_1) -- (G1_2);
        \draw[edge style] (G1_2) -- (G1_3);
        \draw[edge style] (G1_2) -- (G1_4);
        \draw[edge style] (G1_3) -- (G1_5);
        \draw[edge style] (G1_3) -- (G1_6);
        \draw[edge style] (G1_4) -- (G1_7);
        \draw[edge style] (G1_4) -- (G1_8);
    \end{scope}

    \begin{scope}[xshift=4.5cm]
        \node[node style] (G2_1) at (0, 0) {};
        \node[node style] (G2_2) at (0, -1.5) {};
        \node[node style] (G2_3) at (-1, -3) {};
        \node[node style] (G2_4) at (1, -3) {};
        \node[node style] (G2_5) at (-1.5, -4.5) {};
        \node[node style] (G2_6) at (0, -4.5) {};
        \node[node style] (G2_7) at (1.5, -4.5) {};

        \draw[edge style] (G2_1) -- (G2_2);
        \draw[edge style] (G2_2) -- (G2_3);
        \draw[edge style] (G2_2) -- (G2_4);
        \draw[edge style] (G2_3) -- (G2_5);
        \draw[edge style] (G2_3) -- (G2_6);
        \draw[edge style] (G2_4) -- (G2_6);
        \draw[edge style] (G2_4) -- (G2_7);
    \end{scope}

    \begin{scope}[xshift=9cm]
        \node[node style] (G3_1) at (0, 0) {};
        \node[node style] (G3_2) at (0, -1.5) {};
        \node[node style] (G3_3) at (-1, -3) {};
        \node[node style] (G3_4) at (1, -3) {};
        \node[node style] (G3_5) at (-1, -4.5) {};
        \node[node style] (G3_6) at (1, -4.5) {};

        \draw[edge style] (G3_1) -- (G3_2);
        \draw[edge style] (G3_2) -- (G3_3);
        \draw[edge style] (G3_2) -- (G3_4);
        \draw[edge style] (G3_3) -- (G3_5);
        \draw[edge style] (G3_3) -- (G3_6);
        \draw[edge style] (G3_4) -- (G3_5);
        \draw[edge style] (G3_4) -- (G3_6);
    \end{scope}
\end{tikzpicture}
\caption{Forbidden subgraphs in Proposition~\ref{prp:delta=3-triangle-free}}
\label{fig:forbidden}
\end{figure}
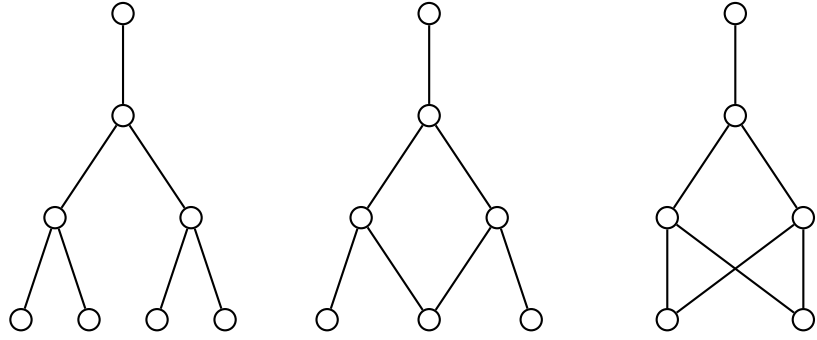

\begin{proof}
Assume first that $\sab(G) = 3$. Then, as already observed, $\Delta(G)\ge 3$. Moreover, in each of the graphs from Fig.~\ref{fig:forbidden}, Runner has a strategy to visit four vertices by first moving along the edge between the top vertex of a corresponding graph and its (unique) neighbor. Hence, if $G$ contained any of the three graphs from the figure as a subgraph, Lemma~\ref{lem:subgraph} would imply that $\sab(G)\ge 4$.

Assume now that $\Delta(G) \ge 3$ and that $G$ contains no subgraph isomorphic to one of the graphs from Fig.~\ref{fig:forbidden}. Since $\Delta(G) \ge 3$, we have $\sab(G)\ge 3$. By way of contradiction, suppose that $\sab (G)\ge 4$. Let $uv$ be the first edge traversed by Runner which enables her to visit at least four vertices during the sabotage game. Then $\deg_G(v)\ge 3$; otherwise, $v$ has at most one other incident edge, which Blocker immediately removes, trapping Runner at $v$. Let $v_1$ and $v_2$ be two neighbors of $v$ different from $u$. As we have supposed that $\sab (G)\ge 4$, and since Blocker can remove $vv_1$ or $vv_2$ in his first move, each of $v_1$ and $v_2$ must have two additional neighbors different from $v$. Let $v_1'$ and $v_1''$ be two such neighbors of $v_1$, and $v_2'$ and $v_2''$ be two such neighbors of $v_2$. Since $G$ is triangle-free, $\{v_1', v_1''\} \cap \{v_2, u\} = \emptyset$ and $\{v_2', v_2''\} \cap \{v_1, u\} = \emptyset$. Setting $t = |\{v_1', v_1'', v_2', v_2''\}|$, we have $2\le t\le 4$. Depending on the value of $t$, we find one of the graphs from Fig.~\ref{fig:forbidden} as a subgraph of $G$, a contradiction. We can conclude that $\sab(G) = 3$.
\end{proof}

To complete the section, we introduce the following relevant concept for the sabotage game. Let $G$ be a connected graph. The \emph{lower sabotage number} of $G$ is defined by 
$$\sab^{-}(G) = \max\{\sab(T):\ T \text{ is a spanning tree of } G \}\,.$$
The name of this concept is justified by the following immediate consequence of Lemma~\ref{lem:subgraph}. 

\begin{corollary}
\label{cor:sab-versus-sab-minus}
If $G$ is a connected graph, then $\sab^{-}(G) \le \sab(G)$.
\end{corollary}

Clearly, if $T$ is a tree, then $\sab^{-}(T) = \sab(T)$. On the other hand, as will be demonstrated in Proposition~\ref{prp:complete-graphs-lower}, $\sab^{-}(G)$ can be arbitrarily smaller than $\sab(G)$.

\section{Trees and unicyclic graphs}
\label{sec:trees}

\begin{theorem}
\label{thm:trees}
If $T$ is a tree with at least three vertices, then 
\[
\sab(T)=\max\{\height(T_u) + 1 + [\deg_{T}(u)>2] : \ T_u \text{ complete binary subtree of } T\}.
\]
\end{theorem}

\begin{proof}
Set 
\[
M = \max\{ \height(T_u) + 1 + [\deg_T(u) > 2] : T_u \text{ complete binary subtree of } T \}.
\]
We will first show $\sab(T) \ge M$. To prove this, we must provide a strategy for Runner which guarantees a path of at least $M$ vertices. 

Let $u \in V(T)$ and $T_u$ be a complete binary subtree rooted at $u$ of height $h = \height(T_u)$ that achieves the maximum value $M$. We consider two cases based on the degree of $u$.

Assume $\deg_T(u) \le 2$. In this case, the formula evaluates to $h + 1$, and Runner uses the following strategy. Because $T_u$ is a complete binary subtree of height $h \ge 1$ (since $T$ has at least three vertices) and $\deg_T(u) \le 2$, $u$ has exactly two children and no other neighbors in $T$. Runner starts the game by selecting and moving along an edge from $u$ to one of its children, say $v$. Runner has now visited two vertices ($u$ and $v$), and $v$ is the root of a complete binary subtree of height $h-1$. 

Blocker then removes exactly one edge $e$. Hence, $e$ can belong to at most one of the two branches of $T_v$ containing the two complete binary subtrees of height $h-2$. Runner then moves to the root of the subtree to which $e$ does not belong. If $e$ is not in any of them, she can arbitrarily choose a branch. 

By repeating this process, namely always moving into the branch that does not contain the edge chosen by Blocker, Runner can successfully make $h$ total moves. Including the starting vertex $u$, Runner visits $h + 1$ vertices, satisfying the bound.

Next, assume that $\deg_T(u) > 2$. In this case, the formula evaluates to $h + 2$ and Runner uses the following strategy. Since $T_u$ contains exactly two edges incident to $u$, and $\deg_T(u) > 2$, there exists a vertex $w$ adjacent to $u$ such that $w\not\in V(T_u)$. Runner starts the game by moving along the edge $wu$. 

Blocker then removes exactly one edge $e$. As in the previous case, $e$ can belong to at most one of the two branches containing the two complete binary subtrees of height $h-1$. Runner then moves to the root of a complete binary subtree of height $h-1$ to which $e$ does not belong. 

Following the same line of thought as in the previous case, Runner can now move at least $h-1$ more times following this strategy. Since the total number of edges Runner traverses is at least $1 + 1 + (h-1)$, she visits at least $h + 2$ vertices.

By the above, $\sab(T) \ge M$. 

Now, let us prove that $\sab(T)\le M$. 

For any edge $xy \in E(T)$, let $T_{x \to y}$ denote the connected component of $T - xy$ containing $y$. Let $g(y, T_{x \to y})$ be the maximum number of additional edges Runner can guarantee to traverse within $T_{x \to y}$ starting from $y$. Because Runner maximizes her path length and Blocker minimizes it by removing exactly one adjacent edge. Let $N^*(y)$ be the set of neighbors of $y$ in $T_{x \to y}$. We define the multiset of available subtree values:
$$S_y = \{ 1 + g(z, T_{y \to z}):\ z \in N^*(y) \}.$$
If $|S_y| \le 1$, Blocker removes the only available edge (if any), leaving Runner with no moves. Thus, $g(y, T_{x \to y}) = 0$. 
If $|S_y| \ge 2$, Blocker removes the edge corresponding to the maximum of $S_y$. Runner optimally chooses the edge corresponding to the remaining maximum. Therefore, $g(y, T_{x \to y})$ is at most the second-largest value in the multiset $S_y$.

We claim that if $g(y, T_{x \to y}) \ge d$, then $T_{x \to y}$ contains a complete binary subtree of height $d$ rooted at $y$. We prove this by induction on $d$. The base case $d=0$ holds trivially, as $y$ forms a complete binary subtree of height $0$. Assume the claim holds for $d-1$. If $g(y, T_{x \to y}) \ge d \ge 1$, then by definition, the second-largest value in $S_y$ is at least $d$. This mathematically implies that there exist at least two distinct vertices $z_1, z_2 \in N^*(y)$ such that $1 + g(z_i, T_{y \to z_i}) \ge d$, which simplifies to $g(z_i, T_{y \to z_i}) \ge d-1$ for all $i \in \{1, 2\}$. By the inductive hypothesis, $T_{y \to z_1}$ and $T_{y \to z_2}$ contain complete binary subtrees of height $d-1$ rooted at $z_1$ and $z_2$, respectively. Because $T$ is a tree, these subtrees are disjoint. Moreover, $z_1$ and $z_2$ are both neighbors of $y$, hence $y$ is the root of a complete binary subtree of height $d$, proving the claim.

Let $k = \sab(T)$. Since $T$ has at least $3$ vertices, it contains $P_3$ as a subgraph. The center of $P_3$ roots a complete binary subtree of height $1$ and has degree at least $2$. Thus, $M \ge 1 + 1 + 0 = 2$. If $k \le 2$, then $\sab(T) \le M$ trivially holds. 
Assume $k \ge 3$. Let Runner's optimal first move be along the edge $v_1v_2$. The remainder of the game is confined to $T_{v_1 \to v_2}$. To visit exactly $k$ vertices, Runner traverses a total of $k-1$ edges. After traversing $v_1v_2$, she must traverse exactly $k-2$ additional edges. Therefore: 
$g(v_2, T_{v_1 \to v_2}) \ge k-2$.
As proven above, $T_{v_1 \to v_2}$ contains a complete binary subtree of height at least $k-2$ rooted at $v_2$. Let this subtree be denoted $B$. Thus, $\height(B) \ge k-2$. 
Because $k \ge 3$, we have $k-2 \ge 1$. Consequently, $v_2$ connects to at least two children within $B \subseteq T_{v_1 \to v_2}$. Furthermore, in the original tree $T$, $v_2$ is also adjacent to $v_1 \notin V(T_{v_1 \to v_2})$, hence $\deg_T(v_2) \ge 3$.
By the definition of $M$, evaluating the specific complete binary subtree $B$ yields:
$$M \ge \height(B) + 1 + [\deg_T(v_2) > 2] \ge (k-2) + 1 + 1 \ge k.$$
Thus, $\sab(T) \le M$, completing the proof.
\end{proof}

Theorem~\ref{thm:trees} allows us to prove the following result. 

\begin{proposition}
\label{prp:complete-graphs-lower}
If $n\ge 2$, then $\sab^{-}(K_n) = \lfloor \log_2(n) \rfloor + 1$.
\end{proposition}

\begin{proof}
Let $n\ge 2$ and let $k$ be the height of a largest complete binary subtree of $K_n$, and fix $T'$ to be such a tree. Then $T'$ is of order $2^{k+1} - 1$, where clearly $2^{k+1} - 1 \le n$. Consider the following two cases. 

\medskip\noindent
{\bf Case 1}: $2^{k+1} - 1 = n$.\\
In this case, $T:=T'$ is a spanning tree of $K_n$ and $k = \log_2(n+1) - 1$. By Theorem~\ref{thm:trees}, $\sab(T) = k + 1 = \log_2(n+1)$.    

\medskip\noindent
{\bf Case 2}: $2^{k+1} - 1 < n$.\\
In this case, we construct a spanning tree $T$ from $T'$, which is a rooted subtree of $K_n$. Let $u$ be its root. Because $n > 2^{k+1} - 1$, there is at least one remaining vertex. We connect exactly one of these remaining vertices to the root $u$ of $T'$. Any other remaining vertices can be attached as leaves to this new vertex so they do not form larger binary subtrees. Hence, by Theorem~\ref{thm:trees}, $\sab(T) = k + 2 <  \log_2(n+1) + 2$. It follows that $\sab(T) = \lfloor \log_2(n+1)\rfloor + 1$.

\medskip
Combining the two cases, we have 
$$\sab(T) =  
\begin{cases}
\log_2(n+1); & n = 2^{k+1} - 1\,,\\[5pt]
\lfloor \log_2(n+1)\rfloor + 1; & \text{otherwise}\,.
\end{cases}$$ 
From here we can conclude that $\sab^{-}(K_n) = \lfloor \log_2(n)\rfloor + 1$.
\end{proof}

\begin{theorem}\label{thm:unicyclic}
If $G$ is a unicyclic graph, then 
$\sab^{-}(G) \leq \sab(G) \leq  \sab^{-}(G)+1$. Moreover, both bounds are sharp.
\end{theorem}

\begin{proof}
Let $G$ be a unicyclic graph and $C$ be the unique cycle in $G$. The lower bound follows by Corollary~\ref{cor:sab-versus-sab-minus}. 

To prove the upper bound, let $u_1 u_2$ be Runner's first move. Blocker then removes an edge $e\in E(C)$ that Runner has not just traversed. Since $C$ has at least 3 edges, such an edge always exists. The graph $G-e$ is a spanning tree, say $T_e$. No matter how Runner plays from $u_2$, she is restricted to traversing the edges of $T_e$. Therefore, in the remainder of the game, she can visit at most $\sab(T_e)$ additional vertices. Together with her initial vertex $u_1$, she visits at most $\sab(T_e) + 1$ vertices. By the definition of the lower sabotage number, $\sab(T_e) \le \sab^{-}(G)$. Therefore, 
$$\sab(G) \le \sab(T_e) + 1 \le \sab^{-}(G) + 1\,.$$

As already mentioned, $\sab(C_n)=2$ for any cycle $C_n$, with $n\geq 3$. Also, the highest complete binary tree a cycle $C_n$ contains is of height 1 and its root is always of degree 2. By Theorem \ref{thm:trees}, $\sab(T)=2$ for any spanning tree $T$ of $C_n$. Hence, $\sab^{-}(C_n) = 2 = \sab(C_n)$ and therefore cycles achieve the lower bound.

For the upper bound sharpness, consider the following graph. Let $H$ be a complete binary tree of height $h\geq 2$ with the root $r$ and an additional leaf $u$ adjacent to $r$. Construct $G$ by identifying the right most leaf of the left subtree of $r$ (ignoring $u$) with the leftmost leaf of the right subtree of $r$ (ignoring $u$). Clearly, $G$ is a unicyclic graph with a unique cycle $C$. 

Let $T$ be a spanning tree of $G$ such that $\sab(T)=\sab^{-}(G)$. $T$ is obtained from $G$ by removing an edge $e\in E(C)$. Then for any $v\in V(T)$ and any complete binary subtree $T_v$ of $T$ it holds that $\height(T_v) \leq h-1$. By Theorem \ref{thm:trees}, $\sab(T) = h - 1 + 2 = h+1$. Hence, $\sab^{-}(G) = h+1$.

Consider the following strategy for Runner on $G$. She imagines playing the game in the original $H$ and responds by moving via the corresponding edge in $G$. In her first move she starts at $u$ and moves to $r$. Next, Blocker removes some edge $e$ of $G$. She imagines Blocker removing the corresponding edge in $H$, uses her optimal strategy there, and moves via the corresponding available edges in $G$. It follows that she can make $h$ more moves, hence the total number of vertices she visits is $h+2 = \sab^{-}(G)+1$. This demonstrates that the upper bound is sharp.
\end{proof}

\section{Subcubic graphs}
\label{sec:subcubic}

The following concept from a very recent paper~\cite{lin-2026} will be extremely useful to us in this section. Let $G$ be a bridgeless, connected graph. Then every edge $e$ lies in a cycle. Let $\ell_G(e)$ be the length of a shortest cycle containing $e$. Then 
$$g^\ast(G) = \max \{ \ell_G(e):\ e\in E(G)\}$$
is the {\em maximum edge girth} of $G$.

\begin{theorem}
\label{thm:subcubic}
If $G$ is a connected, subcubic, bridgeless graph, then $\sab(G) \le g^\ast(G)$. Moreover, the bound is sharp.
\end{theorem}

\begin{proof}
Let $u_1u_2$ be Runner's optimal first move, and let $C = u_1u_2\ldots u_ku_1$ be a shortest cycle of $G$ containing $u_1u_2$. Since $G$ is bridgeless, such a cycle exists. Note that $k\le g^\ast(G)$. 

Assume first that every vertex of $C$ is of degree $3$, and let $x_1, \dots, x_k$ be the respective neighbors of $u_1, \dots, u_k$ which do not belong to $C$. Consider the following strategy for Blocker. He first removes the edge $u_2x_2$. After that, Runner is forced to traverse $u_2u_3$. Next, Blocker removes the edge $u_3x_3$. Proceeding along these lines, Blocker can achieve the goal that Runner can only traverse the edges of $C$. Moreover, if some of the vertices of $C$ are of degree $2$, Blocker can achieve the same goal. Hence we can conclude that $\sab(G) \le k \le g^\ast(G)$.

For the sharpness, consider the grid graphs $P_2\cp P_n$, $n\ge 4$. (For the definition of the Cartesian product $G\cp H$ of graphs $G$ and $H$ see~\cite{hammack-2011}.) Fix $n$ and set $G = P_2\cp P_n$ for the rest of the proof. Note first that $g^\ast(G) = 4$. Hence by the already proved inequality, $\sab(G) \le 4$. On the other hand, $G$ contains the middle graph from Fig.~\ref{fig:forbidden} as a subgraph, see Fig.~\ref{fig:p2-4-subgraph} for the case when $n=4$. 

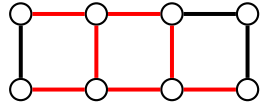
\begin{figure}[ht!]
    \centering
\begin{tikzpicture}[
    node style/.style={circle, draw, fill=white, minimum size=8pt, inner sep=0pt, thick},
    edge style/.style={line width=1.5pt}
]
        \node[node style] (p1-1) at (0, 0) {};
        \node[node style] (p1-2) at (1, 0) {};
        \node[node style] (p1-3) at (2, 0) {};
        \node[node style] (p1-4) at (3, 0) {};

        \node[node style] (p2-1) at (0, 1) {};
        \node[node style] (p2-2) at (1, 1) {};
        \node[node style] (p2-3) at (2, 1) {};
        \node[node style] (p2-4) at (3, 1) {};

        \draw[edge style] (p1-1) -- (p2-1);
        \draw[edge style] (p2-3) -- (p2-4);
        \draw[edge style, red] (p1-2) -- (p2-2);
        \draw[edge style, red] (p1-3) -- (p2-3);
        \draw[edge style] (p1-4) -- (p2-4);
        \draw[edge style, red] (p2-1) -- (p2-2)--(p2-3);
        \draw[edge style, red] (p1-1) -- (p1-2)--(p1-3)--(p1-4);
        
\end{tikzpicture}
\caption{A subgraph of $P_2\cp P_4$}
\label{fig:p2-4-subgraph}
\end{figure}

Proposition~\ref{prp:delta=3-triangle-free} now implies that $\sab(G) \ge 4$. We can conclude that $\sab(G) = 4 = g^\ast(G)$.
\end{proof}

To complete the section we demonstrate that $\sab(G)$ can be arbitrarily smaller than $g^\ast(G)$, that is, the bound in Theorem~\ref{thm:subcubic} can be far from optimal. Let $G_m$, $m \ge 3$, be the cubic graph depicted in Fig.~\ref{fig:subcubic}. It is obtained from a cycle $C_{2m}=u_0u_1\cdots u_{2m-1}u_0$, $m \ge 3$, by attaching, along alternating edges of the cycle, copies of the gadget shown in the figure. 

\begin{figure}[ht!]
    \centering
\begin{tikzpicture}[
    dot/.style={circle, draw, fill=white, inner sep=1.5pt},
    ring edge/.style={draw=black, thick},
    gadget edge/.style={draw=black, thick},
    dotted edge/.style={draw=black, thick, loosely dotted},
    bend arc/.style={draw=black, thick, out=0, in=0,looseness=2},
    scale=0.8
]

\def\radius{3cm}
\def\n{10} 
\def\angleStep{36} 

\node at (0,0) [scale=1.5] {$C_{2m}$};

\foreach \i in {1,...,\n} {
    \pgfmathsetmacro{\angle}{(\i-1)*\angleStep + 18}
    \node[dot] (v\i) at (\angle:\radius) {};
}

    \pgfmathsetmacro{\angle}{(6)*\angleStep + 18}
    \pgfmathsetmacro{\angle}{(5)*\angleStep + 18}
    \node[right] (ui) at (\angle:\radius) {$u_i$};
    \pgfmathsetmacro{\angle}{(4)*\angleStep + 18}
    \node[right] (uip1) at (\angle:\radius) {$u_{i+1}$};
    \pgfmathsetmacro{\angle}{(3)*\angleStep + 18}

    \node[below] (g1) at (-4,-1) {$u'_i$};
    \node[right] (g2) at (-4, 0) {$w_i$};
    \node[above] (g3) at (-4, 1) {$u'_{i+1}$};

\foreach \i in {1,...,\n} {
    \pgfmathtruncatemacro{\j}{mod(\i,\n)+1}
    \ifnum\i=\n
        \draw[dotted edge] (v\i) -- (v\j);
    \else
        \draw[ring edge] (v\i) -- (v\j);
    \fi
}

\foreach \i in {1,...,\n} {
    \pgfmathsetmacro{\angle}{(\i-1)*\angleStep + 18}
    \node[dot] (v\i) at (\angle:\radius) {};
}

\node at (-144:\radius + 0.10cm) {$e$};

\def\placeGadget#1{
    \pgfmathsetmacro{\startIdx}{#1}
    \pgfmathsetmacro{\centerAngle}{(\startIdx-1)*\angleStep + \angleStep/2 + 18}
    
    \begin{scope}[rotate=\centerAngle, shift={(\radius, 0)}]
        \coordinate (b1) at (-0.1, -0.927); 
        \coordinate (b2) at (-0.1, 0.927);
        
        \coordinate (q1A) at (1.0, -0.927);
        \coordinate (q1B) at (1.0, 0.927);
        \coordinate (qMid) at (1, 0); 
        
        \coordinate (q2A) at (2, -0.927);
        \coordinate (q2B) at (2, 0.927);
        \coordinate (q2Mid) at (2, 0);
        
        \draw[gadget edge] (b1) -- (q1A);
        \draw[gadget edge] (q1B) -- (b2);
        \draw[gadget edge] (q1A) -- (qMid) -- (q1B);
        \draw[gadget edge] (q1A)--(q2A)--(q2Mid)--(q2B)--(q1B);
        \draw[gadget edge] (qMid) -- (q2Mid);
        
        \draw[bend arc] (q2B) to (q2A);
        
        \foreach \p in {q1A, q1B, qMid, q2A, q2B, q2Mid} { \node[dot] at (\p) {}; }
    \end{scope}
}

\foreach \edgeStart in {1, 3, 5, 7, 9} {
    \placeGadget{\edgeStart}
}

\foreach \i in {1,...,\n} {
    \pgfmathsetmacro{\angle}{(\i-1)*\angleStep + 18}
    \node[dot] at (\angle:\radius) {};
}

\end{tikzpicture}
\caption{The graph $G_m$, $m \geq 3$}
\label{fig:subcubic}
\end{figure}
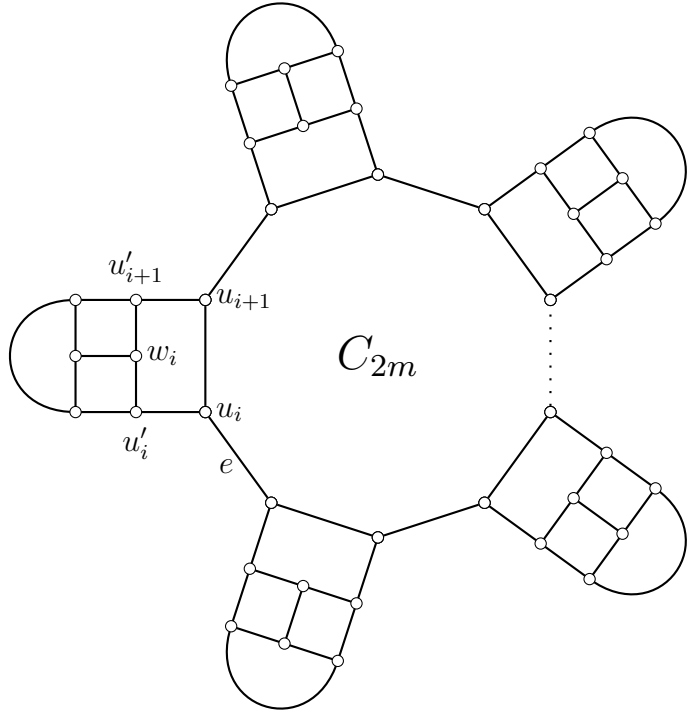

\begin{proposition}
\label{prp:Gm}
If $m\ge 4$, then 
$$g^\ast(G_m) - \sab(G_m) \ge 2m -6\,.$$
\end{proposition}

\begin{proof}
All indices in the proof will be considered modulo $2m$. Denote three vertices of the gadget corresponding to the edge $u_iu_{i+1}$ by $u'_i$, $w_i$, and $u'_{i+1}$, as indicated in Fig.~\ref{fig:subcubic}. Set also $e = u_{i-1}u_i$.

We first determine $g^\ast(G_m)$. The edge $e$ lies on the large cycle $C_{2m}$. By construction, no shorter cycle contains $e$. Therefore $\ell_{G_m}(e)=2m$. On the other hand, every edge different from $e$ lies on a cycle of length at most $2m$. Hence $g^\ast(G_m)=2m$. To complete the argument, we need to demonstrate that $\sab(G_m) \le 6$.

Consider the following strategy of Blocker. Assume first that Runner starts by traversing the edge  $u_{j-1}u_{j}$ of the cycle $C_{2m}$. Then Blocker removes the edge $u_ju_{j+1}$. Since Runner is not allowed to traverse an edge that has already been used, she cannot return to $u_{j-1}$. If Runner has no available move, then the game ends immediately. Otherwise, in her next move she must enter one of the attached gadgets. Without loss of generality, assume that Runner enters the gadget drawn around the edge $u_iu_{i+1}$ by traversing one of the edges $u_i u'_i$ or $u_{i+1} u'_{i+1}$, which are both part of the $5$-cycle $u_iu'_iw_iu'_{i+1}u_{i+1}u_i$. Since the vertices forming this $5$-cycle have degree $3$, Blocker can force Runner to visit only those five vertices. Together with at most one other vertex visited outside of this $5$-cycle by Runner, we have at most six visited vertices by Runner. The same argument applies if Runner starts inside a gadget. In that case, Blocker immediately uses the strategy from the proof of Theorem~\ref{thm:subcubic}: he chooses a shortest cycle in the gadget containing Runner's first edge and, whenever Runner reaches a vertex of that cycle, removes the unique edge leaving the cycle, if such an edge exists. Therefore Runner remains trapped in a cycle of length at most $5$. It follows that in all cases Blocker can ensure that Runner visits at most six vertices, and hence $\sab(G_m)\le 6$. 

We can conclude that
$$g^\ast(G_m) - \sab(G_m) = 2m - \sab(G_m) \ge 2m -6\,$$
as claimed.
\end{proof}

The aim of Proposition~\ref{prp:Gm} was to show that $g^\ast(G) - \sab(G)$ can be arbitrarily large. We add that one could actually prove the equality in the proposition. Indeed, the idea is to assume that Runner starts the game by traversing the edge $u'_iu_i$, and then by case analysis conclude that she can visit six vertices of $G_m$, so that $\sab(G_m) = 6$. 

\section{Complete bipartite and Sierpi\'nski graphs}
\label{sec:additional}

\subsection{Complete bipartite graphs}
\label{sec:complete-bipartite}

\begin{theorem}
\label{thm:complete_bipartite}
If $m \ge n \ge 2$, then
$$\sab(K_{m,n}) \leq m+n-\left\lceil \frac{m-1}{n} \right\rceil - 1.$$
\end{theorem}

\begin{proof}
Let $A=\{u_1,\ldots,u_m\}$ and $B=\{v_1,\ldots,v_n\}$ form the bipartition of $K_{m,n}$, where $m \ge n \ge 2$, and put $q=\left\lceil\frac{m-1}{n}\right\rceil$. 

We prove that $\sab(K_{m,n})\le m+n-q-1$. Without loss of generality, Runner's first move is along the edge $u_1v_1$, in either direction. Thus $u_1$ and $v_1$ are visited after the Runner's first move. Since $n \ge 2$, Blocker chooses the vertex $v_n\in B\setminus\{v_1\}$. His first goal is to ensure that $v_n$ is never visited. Whenever Runner is at the vertex $u_i \in A$, Blocker removes the edge $u_iv_n$, provided that it has not already been removed. If the edge has already been removed, then Blocker can remove any other edge. Since Runner can reach $v_n$ only from a vertex of $A$, this guarantees that Runner never visits $v_n$.

Blocker's second goal is to also protect vertices of $A$ if possible. Let $X$ denote the set of vertices in $A$ that have not yet been visited by Runner and have not yet been protected by Blocker. Initially, $X=A\setminus\{u_1\}$, and therefore $|X|=m-1$.

As long as $|X|>n$, Blocker proceeds as follows. He chooses an arbitrary vertex $u\in X$ and protects it: whenever Runner is at a vertex $v_j\in B$, Blocker removes the edge $v_ju$, provided that it has not already been removed. After all edges incident with $u$ have been removed, the vertex $u$ can no longer be visited by Runner, and Blocker chooses a new vertex from $X$ and repeats the procedure. Observe that during the protection of a vertex $u$, Runner can visit at most $n-1$ new vertices of $A$. Indeed, the vertex $v_n$ is never visited, and hence Runner can enter $A$ from at most the remaining $n-1$ vertices of $B$. Consequently, each protection phase accounts for exactly one protected vertex of $A$ by Blocker and at most $n-1$ additional vertices of $A$ visited by Runner. Thus each completed protection phase accounts for at most $n$ vertices of $X$ (note that in the very first protection phase, the number of newly visited vertices by Runner is at most $n-2$ because we already considered the vertex $u_1$ in her first move; and altogether this is again at most $n-1$ visited vertices before Blocker protects a vertex of $A$). Blocker repeats this process as long as $|X|>n$. After that, we distinguish three cases.

\medskip\noindent
{\bf Case 1}: $|X|=1$. \\
In this case, there is exactly one unvisited vertex left in $A$. Moreover, note that $m-1$ is divisible by $n$. Up until now Blocker has completed exactly $\frac{m-1}{n}$ protection phases. When he deleted the last edge of his last protected vertex, Runner was at a vertex $v \in B$. Therefore, Runner will visit that last unvisited vertex in her next move. Hence Blocker has protected exactly $\frac{m-1}{n}=\left\lceil\frac{m-1}{n}\right\rceil$ vertices of $A$.

\medskip\noindent
{\bf Case 2}: $1<|X|<n$. \\
Now $m-1$ is not divisible by $n$. Before reaching this situation, Blocker has completed $\left\lfloor\frac{m-1}{n}\right\rfloor$
protection phases. Since $|X|>1$, Blocker can choose one additional vertex of $X$ and protect it. Therefore the total number of protected vertices of $A$ equals $\left\lfloor\frac{m-1}{n}\right\rfloor+1=\left\lceil\frac{m-1}{n}\right\rceil$.

\medskip\noindent
{\bf Case 3}: $|X|=n$. \\
In this case, $m$ is divisible by $n$, and hence $m-1$ is not divisible by $n$. Again Blocker has already completed $\left\lfloor\frac{m-1}{n}\right\rfloor$ protection phases. Since $|X|=n$, he can protect one additional vertex from $X$. Consequently, the total number of protected vertices of $A$ is $\frac{m}{n}=\left\lceil\frac{m-1}{n}\right\rceil$.

\medskip

In all the three cases, Blocker ensures that at least $q=\left\lceil\frac{m-1}{n}\right\rceil$ vertices of $A$ remain unvisited by Runner. Together with the vertex $v_n\in B$, which is also never visited, this yields at least $q+1=\left\lceil\frac{m-1}{n}\right\rceil+1$ unvisited vertices. Therefore
$$\sab(K_{m,n}) \le m+n-\left\lceil\frac{m-1}{n}\right\rceil-1\,,$$
which proves the bound.
\end{proof}

The bound in Theorem~\ref{thm:complete_bipartite} is sharp, for instance, $\sab(K_{2,2})=\sab(C_4)=2$, and $\sab(K_{3,2})=3$, since $K_{3,2}$ satisfies the conditions of Proposition~\ref{prp:delta=3-triangle-free}. 
Also, since $K_{3,2}$ is a subgraph of $K_{4,2}$ it follows from Lemma \ref{lem:subgraph} and Theorem \ref{thm:complete_bipartite} that $\sab(K_{4,2})=3$. However, when $m \ge 5$ and $n=2$, the upper bound from Theorem~\ref{thm:complete_bipartite} yields
$$m+2-\left\lceil \frac{m-1}{2} \right\rceil - 1 > 3\,,$$ and therefore is not sharp, which can be seen from Theorem~\ref{thm:complete_bipartite_new}. 

We next prove a different upper bound on $\sab(K_{m,n})$ and give an exact result when $n=2$. 

\begin{theorem}
\label{thm:complete_bipartite_new}
If $m \ge 2$, then
$$\sab(K_{m,2}) =  
\begin{cases}
2; & m=2\,,\\[5pt]
3; & m \ge 3\,,
\end{cases}$$
and if  $m \geq n \ge 3$, then
$$\sab(K_{m,n}) \leq m+n-\left\lfloor \frac{2m+1}{n+2} \right\rfloor .$$
\end{theorem}

\begin{proof}
Let $A=\{u_1,\ldots,u_m\}$ and $B=\{v_1,\ldots,v_n\}$ be the bipartition of $K_{m,n}$, where $m\ge n\ge 2$.

We first consider the case $n=2$. If $m=2$, then $\sab(K_{2,2})=\sab(C_4)=2$. On the other hand, if $m\ge 3$, then $K_{m,2}$ contains $K_{3,2}$ as a subgraph, and by Proposition~\ref{prp:delta=3-triangle-free} and Lemma~\ref{lem:subgraph}, $\sab(K_{m,2})\ge \sab(K_{3,2})=3$. For the reverse inequality, assume without loss of generality that Runner first traverses the edge $u_1v_1$. Blocker removes the edge $u_1v_2$. Since Runner is not allowed to traverse the edge $u_1v_1$ again, she cannot move back to $u_1$. Thus, she must move from $v_1$ to some unvisited vertex $u_i\in A$. Blocker then removes the edge $u_iv_2$. Since the edge $v_1u_i$ has already been traversed by Runner, she has no available edge left. Hence Runner
visits at most three vertices, and so $\sab(K_{m,2})=3$ for every $m\ge 3$.

Assume now that $n\ge 3$. Without loss of generality, Runner first traverses
the edge $u_1v_1$ (in some direction). Blocker chooses an unvisited vertex
$u_j\in A\setminus\{u_1\}$ and starts to protect it. Whenever it is Blocker's turn,
he removes an edge incident with $u_j$. If Runner is currently at a vertex
$v_k\in B$ and the edge $v_ku_j$ is still present, then Blocker removes precisely
this edge. Thus Runner cannot enter $u_j$ in her next move. Otherwise, Blocker
removes an arbitrary remaining edge incident with $u_j$. Blocker continues until all edges incident with $u_j$ have been removed. Notice that once Blocker starts protecting $u_j$, Runner can never visit $u_j$. After $u_j$ is isolated, Blocker chooses another unvisited vertex of $A$ and repeats the same strategy. In addition, if Runner ever enters a vertex of degree at most $2$, then Blocker removes the remaining available edge from that vertex, if such an edge exists. Since Runner is not allowed to traverse the edge by which she entered this vertex again, this ends the game.

Let $V$ be the set of vertices from $A\setminus\{u_1\}$ visited by Runner before
Blocker ever starts protecting them, and let $P$ be the set of vertices of $A$ which
Blocker starts protecting before Runner visits them. By Blocker's strategy described above, no vertex of $P$ is ever visited by Runner.

Suppose that $|V|=r$. Before Runner can add the first vertex to $V$, and since $n \ge 3$, Blocker has already made one move towards protecting some vertex of $A$. Moreover, each time
Runner visits a vertex of $V$ and later reaches a new vertex of $A$ again, Blocker gets two moves in between: one after Runner enters the vertex of $V$, and one after Runner leaves it. Hence, by the time Runner has created $r$ vertices in $V$, Blocker has had at least
$1+2r$ moves available for protecting vertices of $A$. Since one protected vertex of $A$ requires exactly $n$ deleted incident edges, we have
$$
|P|\ge \left\lfloor\frac{1+2r}{n}\right\rfloor .
$$
Therefore, by accounting the vertex $u_1 \in A$, which was visited at the very start of the game,
$$
1+r+\left\lfloor\frac{1+2r}{n}\right\rfloor \le 1 + |V| + |P| \le m.$$
It remains to determine the largest integer $r$ satisfying this inequality. The above inequality is equivalent to
$$\left\lfloor\frac{1+2r}{n}\right\rfloor \le m-1-r.$$
Since the right-hand side is an integer, it follows that
$$\frac{1+2r}{n}<m-r\,,$$
equivalently
$$r<\frac{nm-1}{n+2}.$$
Consequently, the largest possible integer $r_{\max}$ satisfying this inequality is
\begin{align*}
r_{\max} &=\left\lceil\frac{nm-1}{n+2}\right\rceil-1\\
        &=\left\lceil\frac{m(n+2)-(2m+1)}{n+2}\right\rceil-1\\
		&=\left\lceil m-\frac{2m+1}{n+2}\right\rceil -1\\
        &=m-\left\lfloor\frac{2m+1}{n+2}\right\rfloor-1\,.
\end{align*}
Runner can visit at most the initial vertex $u_1$, the vertices of $V$, and all $n$ vertices of $B$. Hence
$$\sab(K_{m,n}) \le 1+r_{\max}+n = m+n-\left\lfloor\frac{2m+1}{n+2}\right\rfloor.$$   
\end{proof}

For the upper bounds of Theorems~\ref{thm:complete_bipartite}
and~\ref{thm:complete_bipartite_new}, computer calculations indicate that the bound of Theorem~\ref{thm:complete_bipartite} is stronger when the two parts of the partition are of similar size, i.e., when $m$ is close to $n$. On the other hand, the bound of
Theorem~\ref{thm:complete_bipartite_new} becomes significantly stronger when $m$ is substantially larger than $n$. This is because the two bounds arise from different Blocker strategies, depending on whether Blocker focuses on protecting a vertex of the partition part $B$ or instead only protects vertices of the partition part $A$. Since the two bounds are independent of each other, we state the following combined upper bound.

\begin{corollary}
\label{cor:complete_bipartite_both}
If $m \ge n \ge 2$, then
$$\sab(K_{m,n}) \le m + n - \max\left\{\left\lceil\frac{m-1}{n}\right\rceil+1,\, \left\lfloor\frac{2m+1}{n+2}\right\rfloor \right\}.$$
\end{corollary}

\subsection{(Generalized) Sierpi\'nski graphs}
\label{sec:sierpinski}

Let $G$ be a connected graph of order $n(G)=p$ with vertex set $V(G)=[p]_0$. For a positive integer $k$, the generalized Sierpi\'nski graph $S^k_G$ is the graph whose vertex set is $V(S^k_G)=[p]^k_0$. Thus every vertex of $S^k_G$ is a $k$-tuple, where each element is from $[p]_0$. Let $s=s_k\cdots s_1$ and $t=t_k\cdots t_1$ be two distinct vertices of $S^k_G$. Then $s$ and $t$ are adjacent if there exists an index $d \in [k]$ such that

\begin{enumerate}
\item $s_\delta=t_\delta$ for every $\delta>d$;
\item $s_d\neq t_d$ and $s_dt_d \in E(G)$;
\item $s_\delta=t_d$ and $t_\delta=s_d$ for every $\delta<d$.
\end{enumerate}

\noindent The edge set of $S^k_G$ can also be written in the following way:
$$E(S^k_G)=\left\{\{\underline{s}ij^{d-1},\,\underline{s}ji^{d-1}\} :\ d\in[k],\; \underline{s}\in[p]_0^{\,k-d},\; \{i,j\}\in E(G)\right\}.$$

Observe that $S_G^k$ contains $p^{k-1}$ copies of $G$. More precisely, for every
$\underline{s}\in [p]_0^{\,k-1}$, the set $\{\underline{s}i :\ i\in [p]_0\}$ induces a subgraph of $S_G^k$ isomorphic to $G$. We will denote this copy of $G$ by $\underline{s}G$. An edge of $S_G^k$ which does not lie in any $G$-copy will be called a \emph{linking edge}.

The Sierpi\'nski graph $S^k_p$ is defined as the generalized Sierpi\'nski graph whose base graph is the complete graph $K_p$. That is, $S^k_p=S^k_{K_p}$. The classical Sierpi\'nski graphs were introduced in~\cite{klavzar-1997}, see also the survey~\cite{hinz-2017}, and extended to the general case in~\cite{gravier-2011}, see also~\cite{menon-2025} and references therein. 

\begin{proposition}
\label{prp:generalized-sierpinski-bounds}
If $G$ is a connected graph and $k \geq 1$, then 
$$\sab(G) \leq \sab(S^k_G) \leq n(G)+1\, ,$$
and both bounds are sharp.
\end{proposition}

\begin{proof}
If $k=1$, then $S^1_G=G$, and therefore $\sab(S_G^1)=\sab(G)$. Hence the assertion holds in this case.

Assume now that $k\ge 2$. The lower bound follows directly from Lemma~\ref{lem:subgraph}, since
$G$ is an induced connected subgraph of $S_G^k$.

We now prove the upper bound. Let $V(G)=[p]_0$, where $p=n(G)$. Recall that, for every
$\underline{s}\in[p]_0^{k-1}$, the subgraph $\underline{s}G$ is an induced isomorphic copy of $G$ in $S_G^k$, and we have $p^{k-1} > 1$ such copies. By the definition of $S_G^k$, every vertex is incident with at most one linking edge.

Blocker uses the following strategy. Whenever Runner is currently at a vertex
of some copy $\underline{s}G$, Blocker removes the linking edge incident with
Runner's current vertex, provided that such an edge exists. If no such edge exists,
Blocker removes an arbitrary edge.

Assume first that Runner's first move is along an edge contained in some copy
$\underline{s}G$. Then, after this move, Runner is in $\underline{s}G$. By Blocker's
strategy, whenever Runner is at a vertex of $\underline{s}G$, the only possible linking edge by which Runner could leave $\underline{s}G$ is removed before she can use it. Hence Runner remains inside $\underline{s}G$ for the rest of the game. Therefore
Runner visits at most $n(G)$ vertices in this case.

Assume next that Runner's first move is along a linking edge. Then this edge
joins two distinct copies of $G$. After the first move, Runner is in one of these
copies, say $\underline{s}G$, while the other endpoint of the first edge is the only
vertex already visited outside $\underline{s}G$. From this point on, Blocker applies
the same strategy and prevents Runner from leaving $\underline{s}G$. Hence Runner
can visit at most all vertices of $\underline{s}G$, together with the one vertex
visited before entering $\underline{s}G$. Thus Runner visits at most $n(G)+1$
vertices.

Combining the two cases, Blocker can always ensure that Runner visits at most
$n(G)+1$ vertices. Therefore $\sab(S_G^k)\le n(G)+1$, which proves the assertion.

The lower bound is sharp for $k=1$, since $S_G^1=G$. Moreover, we can also find a family of graphs $G$ such that the lower bound is sharp for all $k \geq 1$. Let $G=K_{1,n}$, where $n\ge 3$. Since $G$ is a tree, also $S^k_G$ is a tree, and every complete binary subtree of $S^k_G$ has height $1$. Namely, every vertex of $S^k_{K_{1,n}}$ which has degree at least $3$ corresponds to the center of one of the induced copies of $K_{1,n}$, all its neighbors are leaves in their corresponding copies, and their degree outside of the copy they belong to is either $0$ or $1$ if they are the endvertex of a linking edge. On the other hand, $S^k_G$ contains a complete binary subtree of height $1$ rooted at a vertex of degree at least $3$.  Thus, by Theorem~\ref{thm:trees}, $\sab(S_G^k)=1+1+1=3$. Since also $\sab(G)=3$, we obtain $\sab(S_G^k)=\sab(G)$, showing that the lower bound is sharp.

The upper bound is sharp, e.g.\ for Sierpi\'nski graphs $S^k_p$, $k,p\ge 3$, as shown in Theorem~\ref{thm:sierpinski} below.
\end{proof}

\begin{theorem}
\label{thm:sierpinski}
If $S^k_p$, $k\ge 1$ and $p\ge 1$, is a Sierpi\'nski graph, then 
$$\sab(S^k_p) =  
\begin{cases}
p-1; & k=1 \mbox{ and } p \ge 3\,,\\[5pt]
p; & (k\ge 1 \mbox{ and } p \le 2) \mbox{ or } (k=2 \mbox{ and } p=3)\,,\\[5pt]
p+1; & (k=2 \mbox{ and } p \ge 4) \mbox{ or } (k\ge 3 \mbox{ and } p\ge 3)\,.
\end{cases}$$ 
\end{theorem}

\begin{proof}
Assume first that $k=1$ and $p\ge 3$. Then $S_p^1=K_p$, and by Proposition~\ref{prp:path-cycle-complete} we have $\sab(S_p^1)=\sab(K_p)=p-1$.

Next we have the case $k \geq 1$ and $p\le 2$. Then $S^k_p$ is either an edge-less graph for $p=1$ or a path for $p=2$. Hence $\sab(S_p^k)=p$.

Let now $k=2$ and $p=3$. We claim that $\sab(S_3^2)=3$. Runner can clearly visit at least three vertices, since she can start with a linking edge and then move inside the copy of $K_3$ she enters. For the reverse inequality, Blocker uses the following strategy. If Runner starts inside some copy $iK_3$, then Blocker removes the linking edge incident with Runner's current vertex, if such an edge exists (otherwise he removes an arbitrary edge). Continuing in this way, Runner cannot leave this copy and hence visits at most three vertices. Assume now that Runner starts with a linking edge, say from $ij$ to $ji$, where $i,j \in [3]_0$ with $i\ne j$. Let $h$ be the remaining element of $[3]_0$. Blocker removes the edge $\{ji,jh\}$ inside the copy $jK_3$. Then Runner has to move to $jj$ and Blocker removes the edge $\{jj,jh\}$. Runner cannot visit a fourth vertex, and hence $\sab(S_3^2)=3$.

It remains to consider the cases $k=2,\ p\ge 4$ and $k\ge 3,\ p\ge 3$. By Proposition~\ref{prp:generalized-sierpinski-bounds}, we already know that $\sab(S_p^k)\le p+1$. In the remainder of the proof we show that Runner can visit at least $p+1$ vertices.

First assume that $k=2$ and $p\ge 4$. Let $K=0K_p$. In this copy, every vertex except $00$ is incident with a linking edge. Runner starts by traversing a linking edge towards $K$, say the edge $\{10,01\}$. Thus she has visited one vertex outside $K$ and one vertex of $K$. In her next move we can without loss of generality assume she moves inside $K$ to a vertex different from $00$ and $01$. This is possible because $p\ge 4$ and Blocker has removed at most one edge up to this point. Thus Runner has visited one vertex outside $K$ and two vertices of $K$, both vertices of $K$ being incident with linking edges. From now on Runner proceeds as follows. Whenever she is at a vertex of $K$ incident with a linking edge, she checks whether Blocker has removed this linking edge. If Blocker has removed it, then Runner stays inside $K$ and moves to a new vertex of $K$, always leaving the vertex $00$ for last. If Blocker has not removed the linking edge, Runner traverses it and enters another copy $K'$ of $K_p$. If Blocker always removes the linking edge incident with Runner's current vertex, then Runner remains inside $K$ and visits all vertices of $K$, with $00$ visited last. Together with the first vertex outside $K$, this gives $p+1$ visited vertices. Otherwise, Runner traverses an unremoved linking edge and enters another copy $K'$. Before entering $K'$, she has already visited at least three vertices. Up to this moment Blocker has only removed linking edges incident with vertices of $K$, except possibly in his last move. Hence, inside $K'$, at most one edge has been removed. Runner now stays inside $K'$. Since $K'$ is a complete graph with at most one edge removed, Runner can visit at least $p-2$ vertices of $K'$ ($K_{p-1}$ is a connected subgraph of $K_p$ minus an edge, so we can use Lemma~\ref{lem:subgraph}). Therefore she visits at least
$3+(p-2)=p+1$ vertices in total.

Finally assume that $k\ge 3$ and $p\ge 3$. Let $K=0^{k-2}1K_p$. Every vertex of $K$ is incident with a linking edge. Indeed, if $i\ne 1$, then $0^{k-2}1i$ is joined by a linking edge to $0^{k-2}i1$, while $0^{k-2}11$ is
joined by a linking edge to $0^{k-3}100$. Runner starts by traversing a linking edge towards $K$. Then she makes one move inside $K$ to a new vertex of $K$, which is possible since $p\ge 3$ (and Blocker has only removed one edge up to this point). From this point on she uses the same strategy as described in the previous case: if Blocker removes the linking edge incident with her current vertex, she stays inside $K$ and visits a new vertex of $K$; if Blocker does not remove this linking edge, she traverses it and enters another copy $K'$ of $K_p$. If Blocker always removes the relevant linking edge, then Runner visits all vertices of $K$, and together with the first vertex outside $K$ she visits $p+1$ vertices. Otherwise, Runner enters another copy $K'$ after having already visited at least three vertices. As above, at most one edge of $K'$ has been removed before Runner starts playing inside it, and hence she can visit at least $p-2$ vertices of $K'$. Thus she again visits at least $3+(p-2)=p+1$ vertices.

Consequently, $\sab(S_p^k)\ge p+1$ holds in both cases, and together with Proposition~\ref{prp:generalized-sierpinski-bounds}, this yields $\sab(S_p^k)=p+1$.
\end{proof}

\section{Concluding remarks}

In this paper, we introduced a graph-based version of the sabotage game. The game is generally very challenging, and many questions remain unanswered. Let us list a few selected open problems. 

For unicyclic graphs we have shown that the bounds in Theorem \ref{thm:unicyclic} are sharp. Since the difference between the bounds is exactly one, we pose:

\begin{problem}
Characterize unicyclic graphs which achieve the lower (upper) bound specified in Theorem \ref{thm:unicyclic}.
\end{problem}

To provide sharpness examples in Theorem~\ref{thm:subcubic}, we have demonstrated that $\sab(P_2\cp P_n) = 4$ for  $n\ge 2$. In this direction we pose:

\begin{problem}
Determine $\sab(P_m\cp P_n)$ for all $m, n\ge 2$. If that is too difficult, prove a lower bound and an upper bound on $\sab(P_n\cp P_m)$ that are close to each other.      
\end{problem}

The exact value of the sabotage number of complete bipartite graphs remains unknown. Nevertheless, all smaller cases considered so far suggest that the upper bound from Corollary~\ref{cor:complete_bipartite_both} may in fact be exact. This leads to the following problem.

\begin{problem}
Determine the exact value of the sabotage number of complete bipartite graphs. In particular, is it true that for every $m \ge n \ge 2$,
$$\sab(K_{m,n})=m+n-\max\left\{\left\lceil\frac{m-1}{n}\right\rceil+1,\,\left\lfloor\frac{2m+1}{n+2}\right\rfloor\right\}?$$
\end{problem}

\section*{Acknowledgements}

M.J., S.K., A.T.\ were supported by the Slovenian Research and Innovation Agency (ARIS) under the grants  P1-0297, N1-0285, N1-0355, N1-0431, J1-70045. M.M.\ acknowledges the financial support of the Ministry of Science, Technological Development and Innovation of the Republic of Serbia (Grants No.\ 451-03-33/2026-03/ 200125 \& 451-03-34/2026-03/ 200125). A lot of work on this paper has been done during the Workshop on Games on Graphs IV, June 2026, Rogla, Slovenia, the authors thank the Institute of Mathematics, Physics and Mechanics, Ljubljana, Slovenia, and the Faculty of Natural Sciences and Mathematics, University of Maribor, Slovenia, for supporting the workshop. 

\section*{Declaration of interests}
 
The authors declare that they have no conflict of interest. 

\section*{Data availability}
 
Our manuscript has no associated data.

\end{document}